\documentclass{article}

\usepackage{amsfonts}

\usepackage{amsthm}

\usepackage{amssymb}

\usepackage{amsmath}

\usepackage{enumerate}

\usepackage[all]{xy}

\usepackage{graphicx}

\usepackage{hyperref}

\usepackage{verbatim}

\newcommand{\N}{\mathbb{N}}

\newcommand{\C}{\mathbb{C}}

\newcommand{\h}{\mathcal{H}}

\newcommand{\K}{\mathcal{K}}

\newcommand{\B}{\mathcal{B}}

\newcommand{\pr}{\text{prop}}

\newcommand{\Manoa}{M\=anoa}

\newcommand{\Hawaii}{Hawai\kern.05em`\kern.05em\relax i}

\theoremstyle{definition} \newtheorem{cedef}{Definition}[section]

\theoremstyle{definition} \newtheorem{roealg}[cedef]{Definition}

\theoremstyle{definition} \newtheorem{udbg}[cedef]{Definition}

\theoremstyle{plain} \newtheorem{main1}[cedef]{Theorem}

\theoremstyle{plain} \newtheorem{question2}[cedef]{Question}

\theoremstyle{plain} \newtheorem{question}[cedef]{Question}

\theoremstyle{definition} 

\theoremstyle{definition} 

\theoremstyle{remark} 

\theoremstyle{plain} \newtheorem{main}[cedef]{Theorem}

\theoremstyle{plain} \newtheorem{groupc}[cedef]{Corollary}

\theoremstyle{remark} \newtheorem{mainrem}[cedef]{Remarks}

\theoremstyle{definition} \newtheorem{supp}{Definition}[section]

\theoremstyle{definition} 

\theoremstyle{definition} \newtheorem{roealg2}[supp]{Definition}

\theoremstyle{remark} \newtheorem{roeex}[supp]{Examples}

\theoremstyle{remark} 

\theoremstyle{remark} \newtheorem{roerem}[supp]{Remarks}

\theoremstyle{plain} \newtheorem{spatiallem}{Lemma}[section]

\theoremstyle{plain} \newtheorem{jlem}[spatiallem]{Lemma}

\theoremstyle{plain} \newtheorem{jlem2}[spatiallem]{Lemma}

\theoremstyle{plain} \newtheorem{fadthe}{Theorem}[section]

\theoremstyle{definition} \newtheorem{msp}[fadthe]{Definition}

\theoremstyle{plain} \newtheorem{amsp}[fadthe]{Theorem}

\theoremstyle{plain} 

\theoremstyle{definition} 

\theoremstyle{plain} \newtheorem{cloudlem}[fadthe]{Lemma}

\theoremstyle{plain} 

\theoremstyle{remark} 

\theoremstyle{plain} \newtheorem{fadlem2}[fadthe]{Lemma}

\theoremstyle{plain} \newtheorem{cloudbound1}[fadthe]{Lemma}

\theoremstyle{plain} \newtheorem{athe}{Theorem}[section]

\theoremstyle{plain} 

\theoremstyle{definition} 

\theoremstyle{plain} 

\theoremstyle{plain} \newtheorem{cloudbound2}[athe]{Lemma}

\theoremstyle{plain} \newtheorem{suthe}{Theorem}[section]

\theoremstyle{plain} \newtheorem{morcor}[suthe]{Corollary}

\theoremstyle{plain} \newtheorem{groupc2}[suthe]{Corollary}

\theoremstyle{plain} \newtheorem{sulem}[suthe]{Lemma}

\theoremstyle{plain} \newtheorem{sucloud}[suthe]{Lemma}

\theoremstyle{definition} \newtheorem{corcat}{Definition}[section]

\theoremstyle{definition} \newtheorem{roecat}[corcat]{Definiton}

\theoremstyle{definition} \newtheorem{cover}[corcat]{Definition}

\theoremstyle{plain} \newtheorem{covlem}[corcat]{Lemma}

\theoremstyle{definition} 

\theoremstyle{definition} 

\theoremstyle{plain} \newtheorem{equivthe}[corcat]{Theorem}

\title{On Rigidity of Roe Algebras}

\author{J\'{a}n \v{S}pakula\footnote{Corresponding author} ~and Rufus Willett}

\begin{document}

\maketitle

\noindent
\textsc{Mathematical Sciences, University of Southampton, Southampton
SO17 1BJ, United Kingdom}\\
E-mail: \texttt{jan.spakula@soton.ac.uk}\\

\noindent
\textsc{2565 McCarthy Mall, University of \Hawaii\ at \Manoa, Honolulu, HI 96822, USA}\\
E-mail: \texttt{rufus@math.hawaii.edu}

\begin{abstract}
Roe algebras are $C^*$-algebras built using large-scale (or `coarse') aspects of a metric space $(X,d)$.  In the special case that $X=\Gamma$ is a finitely generated group and $d$ is a word metric, the simplest Roe algebra associated to $(\Gamma,d)$ is isomorphic to the crossed product $C^*$-algebra $l^\infty(\Gamma)\rtimes_r\Gamma$.  

Roe algebras are coarse invariants, meaning that if $X$ and $Y$ are coarsely equivalent metric spaces, then their Roe algebras are isomorphic.  Motivated in part by the coarse Baum-Connes conjecture, we show that the converse statement is true for a very large classes of spaces.  This can be thought of as a `$C^*$-rigidity result': it shows that the Roe algebra construction preserves a large amount of information about the space, and is thus surprisingly `rigid'.

As an example of our results, in the group case we have that if $\Gamma$ and $\Lambda$ are finitely generated amenable, hyperbolic, or linear, groups such that the crossed products $l^\infty(\Gamma)\rtimes_r\Gamma$ and $l^\infty(\Lambda)\rtimes_r\Lambda$ are isomorphic, then $\Gamma$ and $\Lambda$ are quasi-isometric.\\

\end{abstract}

\section{Introduction}

This piece asks to what extent certain noncommutative $C^*$-algebras associated to a metric space $X$, called Roe algebras, capture the large scale geometry of $X$.  As such it forms part of the noncommutative geometry program of Connes \cite{Connes:1994zh}.   

For us, `large scale geometry' will mean those aspects of metric geometry that are preserved under coarse equivalence as in the following definition.

\begin{cedef}\label{cedef}
Let $X$ and $Y$ be metric spaces.  A (not necessarily continuous) map $f:X\to Y$ is said to be \emph{uniformly expansive} if for all $R>0$ there exists $S>0$ such that if $x_1,x_2\in X$ satisfy $d(x_1,x_2)\leq R$, then $d(f(x_1),f(x_2))\leq S$.

Two maps $f,g:X\to Y$ are said to be \emph{close} if there exists $C>0$ such that $d(f(x),g(x))\leq C$ for all $x\in X$.

$X$ and $Y$ are said to be \emph{coarsely equivalent} if there exist uniformly expansive maps $f:X\to Y$ and $g:Y\to X$ such that $f\circ g$ and $g\circ f$ are close to the identity  maps on $Y$ and $X$ respectively.  
\end{cedef}

The metric spaces we study will have the following bounded geometry property: natural examples of metric spaces like finitely generated discrete groups with word metrics, and complete Riemannian manifolds (under additional minor hypotheses) satisfy this condition up to coarse equivalence. 

\begin{udbg}\label{udbg}
Let $X$ be a metric space.  $X$ is said to have \emph{bounded geometry} if for all $R\geq0$ there exists $N_R\in\N$ such that for all $x\in X$, the ball of radius $R$ about $x$ has at most $N_R$ elements.

Throughout the remainder of this piece, the word \emph{space} will mean bounded geometry metric space.
\end{udbg}

In this introduction, we will discuss motivations for the problem: firstly, uniform Roe algebras in abstraction; and secondly, Roe algebras, and the coarse Baum-Connes conjecture.

\subsection*{Uniform Roe algebras}

\begin{roealg}\label{roealg}
Let $X$ be a space and let $T=(T_{xy})_{x,y\in X}$ be an $X$-by-$X$ indexed matrix with entries $T_{xy}$ in $\C$.  The \emph{propagation} of $T$ is defined to be
$$
\pr(T):=\sup\{d(x,y)~|~T_{xy}\neq 0\}
$$
(possibly infinite).

The \emph{algebraic uniform Roe algebra\footnote{Also called the \emph{translation algebra of $X$} after Gromov \cite[page 262]{Gromov:1993tr}.}}, denoted $\C_u[X]$ is defined to be the collection of all finite propagation matrices with uniformly bounded entries, equipped with the $*$-algebra structure coming from the usual matrix operations.

The $*$-algebra $\C_u[X]$ admits an obvious (faithful) representation as bounded operators on $l^2(X)$; the corresponding norm completion is denoted $C^*_u(X)$ and called the \emph{uniform Roe algebra of $X$}. 
\end{roealg}

The schematic below represents an element of $\C_u[X]$.\vspace{0.3cm}\\
\setlength{\unitlength}{8mm}
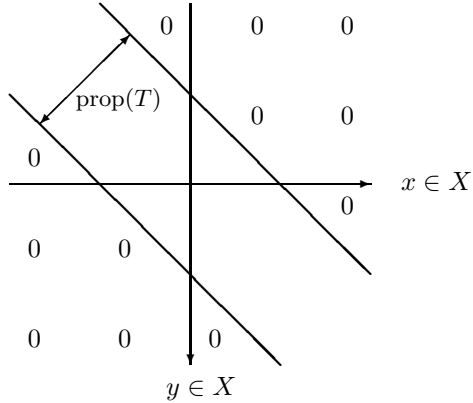
\begin{picture}(12,7)(-8,-4)
\put(-3,0){\vector(1,0){6}}
\put(0,3){\vector(0,-1){6}}
\put(-0.4,-3.5){$y\in X$}
\put(3.5,-0.1){$x\in X$}
\put(-2.5,1){\vector(1,1){1.5}}
\put(-1,2.5){\vector(-1,-1){1.5}}
\put(-1.9,1.3){\small{$\text{prop}(T)$}}
\put(1,1){$0$}
\put(-1.2,-1.2){$0$}
\put(-0.5,2.5){$0$}
\put(2.5,-0.5){$0$}
\put(1,2.5){$0$}
\put(2.5,1){$0$}
\put(-2.7,0.3){$0$}
\put(-2.7,-1.2){$0$}
\put(-1.2,-2.7){$0$}
\put(0.3,-2.7){$0$}
\put(-2.7,-2.7){$0$}
\put(2.5,2.5){$0$}
\thicklines
\put(-3,1.5){\line(1,-1){4.5}}
\put(-1.5,3){\line(1,-1){4.5}}
\end{picture}\\
As we hope is clear from the picture, operators in $\C_u[X]$ can only `move' elements of a space $X$ a bounded distance, and are thus connected to its geometry.  

It is by now well-known that this connection between $C^*$-algebraic properties of $C^*_u(X)$ and geometric properties of $X$ carries quite a lot of information. For example Skandalis--Tu--Yu \cite[Theorem 5.3]{Skandalis:2002ng} and Brown--Ozawa \cite[Theorem 5.5.7]{Brown:2008qy}  have shown that $C^*_u(X)$ is nuclear if and only if $X$ has Yu's property A \cite[Definition 2.1] {Yu:200ve}.  The papers \cite{Guentner:2022hc,Ozawa:2000th,Chen:2004bd,Chen:2005cy,Chen:2008so,Spakula:2009rr,Spakula:2009tg,Winter:2010eb} contain more examples along these lines.

It is thus natural to ask: to what extent does $C^*_u(X)$ (or $\C_u[X]$) determine the coarse equivalence type of $X$?  One of the main results of this paper is the following partial answer.

\begin{main1}\label{main1}
Let $X$ and $Y$ be spaces.  Then:
\begin{itemize}
\item if $\C_u[X]$ and $\C_u[Y]$ are $*$-isomorphic then  $X$ and $Y$ are coarsely equivalent;
\item if $X$ and $Y$ have property A, and if $C^*_u(X)$ and $C^*_u(Y)$ are Morita equivalent (in particular, if they are $*$-isomorphic) then $X$ and $Y$ are coarsely equivalent.
\end{itemize}
\end{main1}
These results follow from Theorems \ref{fadthe}, \ref{athe}, and Corollary \ref{morcor} below.  Note that the first part has a purely algebraic statement in terms of translation algebras.

We will not define property A here as we do not need the usual definition, but instead an equivalent property\footnote{The perhaps less well-known \emph{metric sparsification property} \cite[Definition 3.1]{Chen:2008so}.}; we refer the reader to \cite{Willett:2009rt} for a survey of property A, which has been very widely studied.  Property A holds for a very large class of spaces: notable examples of discrete groups with property A are amenable groups, hyperbolic groups \cite{Roe:2005rt}, and linear groups \cite{Guentner:2005xr}.  The only groups for which property A is known to fail are the random groups of Gromov and Arzhantseva--Delzant whose Cayley graphs coarsely contain a sequence of expanding graphs \cite{Gromov:2003gf,Arzhantseva:2008bv}.

The following purely group theoretic corollary seems to be of interest in its own right.  Quasi-isometry is a stronger version of coarse equivalence that is more widely studied in geometric group theory.

\begin{groupc}\label{groupc}
Say $\Gamma$ and $\Lambda$ are finitely generated $C^*$-exact groups. Then the following are equivalent:
\begin{itemize}
\item $\Gamma$ and $\Lambda$ are quasi-isometric;
\item $l^\infty(\Gamma)\rtimes_r\Gamma$ and $l^\infty(\Lambda)\rtimes_r\Lambda$ are Morita equivalent.
\end{itemize}
If moreover $\Gamma$ and $\Lambda$ are non-amenable, then these are also equivalent to:
\begin{itemize}
\item $l^\infty(\Gamma)\rtimes_r\Gamma$ and $l^\infty(\Lambda)\rtimes_r\Lambda$ are $*$-isomorphic.
\end{itemize}
\end{groupc}

It might be interesting to give a direct $C^*$-algebraic proof of the equivalence of the second and third points above (the authors have no idea how to approach this).

There is an analogy between our results in coarse geometry / $C^*$-algebras and recent deep results of Popa and others (see \cite{Popa:2007fk} for a survey) in the theory of group actions /  von Neumann algebras.  One of the questions these authors consider is as follows: if the group-measure-space von Neumann algebras $L^\infty([0,1]^\Gamma)\overline{\rtimes}\Gamma$ and $L^\infty([0,1]^{\Lambda})\overline{\rtimes}\Lambda$ associated to the Bernoulli shift actions of discrete groups $\Gamma$, $\Lambda$ are isomorphic, then are the actions conjugate (and in particular, are the groups isomorphic)?  This question has (at least) formal similarities to Theorem \ref{groupc} above; moreover, Thomas Sinclair has pointed out to us that there are some interesting parallels between our work and that of Ioana \cite{Ioana:2010uq} on von Neumann algebras associated to Bernoulli shift actions as above.

\subsection*{Roe algebras and the coarse Baum-Connes conjecture}\

The above gives an outline of our main results in the abstract.  The main motivation does not come from abstract $C^*$-algebraic considerations, however, but from the coarse Baum-Connes conjecture; we will now briefly discuss this.

This conjecture involves the \emph{Roe algebra} $C^*(X)$ of a space $X$; this is defined in the same way as $C^*_u(X)$, but the matrix entries are allowed to be compact operators rather than complex numbers - see Section  \ref{defsec} below for a precise definition.  This use of the compact operators makes $C^*(X)$ less tractable than $C_u^*(X)$ as a $C^*$-algebra, but rather more tractable $K$-theoretically.  

Roe algebras were originally introduced by John Roe for the purpose of studying index theory on non-compact manifolds\footnote{See \cite{Roe:1988qy,Roe:1993lq} for the original articles, and \cite[Chapter 3]{Roe:1996dn} and \cite[Chapter 4]{Roe:2003rw} for more recent treatments.}; this program led to the coarse Baum-Connes conjecture (amongst other things), which asks whether a certain assembly map 
\begin{equation}\label{cbc}
\mu:\lim_{R\to\infty}K_*(P_R(X))\to K_*(C^*(X))
\end{equation}
is an isomorphism.  The coarse Baum-Connes conjecture has been very widely studied, and is responsible for some of the best results on problems such as the Novikov conjecture in high-dimensional topology, and the existence of positive scalar curvature metrics.  The original sources are the articles of Roe \cite[Conjecture 6.30]{Roe:1993lq}, Higson--Roe \cite{Higson:1995fv}, and Yu \cite{Yu:1995bv}.  

For the present discussion, the left hand side of line \eqref{cbc} above should be thought of as the `large-scale algebraic topology of $X$' and the right hand side as the `algebraic topology of $C^*(X)$'.  The conjecture asks whether certain natural geometric invariants (to be thought of as coming from elliptic differential operators on $X$) living in the left hand side are preserved by passage to the right hand side, where they have much better properties (e.g.\ vanishing properties, and homotopy invariance).  In other words, the coarse Baum-Connes conjecture asks the following question. 
 
\begin{question2}\label{q2}
If $X$ and $Y$ are metric spaces, how well does the algebraic topology of $C^*(X)$ model the large scale algebraic topology of $X$?
\end{question2}

This motivates the following rigidity question.  It can be thought of as a more precise analogue of Question \ref{q2} above: it asks about information at the level of algebras and spaces themselves, rather than merely at the level of algebraic-topological invariants associated to these algebras and spaces.   A version of this question was also asked by Winter and Zacharias \cite[page 3]{Winter:2010eb}. 

\begin{question}\label{q1}
If $X$ and $Y$ are spaces such that $C^*(X)$ and $C^*(Y)$ are $*$-isomorphic, are $X$ and $Y$ coarsely equivalent?
\end{question}

Question \ref{q1} is called a `rigidity question' as it asks if the $C^*$-algebra $C^*(X)$ actually encodes \emph{all} of the large scale geometry of $X$, and is thus surprisingly `rigid'.  Our main results also apply in this context.

\begin{main}\label{main}
Let $X$ and $Y$ be spaces with property A.  If $C^*(X)$ and $C^*(Y)$ are $*$-isomorphic,  then $X$ and $Y$ are coarsely equivalent.
\end{main}
An analogue of the purely algebraic part of Theorem \ref{main1} also holds in this context: see Theorems \ref{fadthe} and  \ref{athe} below. 

We conclude this introduction with some remarks on Theorems \ref{main1} and \ref{main}. 

\begin{mainrem}\label{mainrem}
\begin{itemize}
\item There are several natural variants of uniform Roe  and Roe algebras that have been introduced for the study of index theory and coarse geometry: see Examples \ref{roeex} below.  Analogues of Theorem \ref{main} and Theorem \ref{main1} apply in these cases.
\item It follows from our methods that, up to a natural notion of equivalence in each case, coarse equivalences of spaces are essentially the same thing as spatially implemented isomorphisms of (algebraic) Roe algebras.  This is made precise in Appendix \ref{cat}.
\item The coarse Baum-Connes conjecture is false in general: see \cite[Section 8]{Yu:1998wj} and \cite{Higson:1999km,Higson:2002la}.  Our work here is a step towards a more precise geometric understanding of why one should expect it to be true in `good' cases.  We hope to be able to expand on this project in order to elucidate what goes wrong in more exotic cases.  A promising direction here seems to be provided by the ghost ideals studied in \cite{Chen:2004bd,Chen:2005cy,Wang:2007uf} and \cite[Section 11.5.2]{Roe:2003rw}.
\end{itemize}
\end{mainrem}

\subsection*{Outline of the piece}

Section \ref{defsec} gives definitions and notation.  The work below will actually be carried out in a more general setting, using what we call `Roe-type algebras': the extra generality is not a goal in its own right, but is useful as it allows various special cases to be covered in a common framework; we hope it also makes the structure of the proofs more transparent.  

Section \ref{apsec} covers some analytic preliminaries: the first shows that $*$-isomorphisms between Roe-type algebras are spatially implemented; the second shows that `cancellation of arbitrarily large propagation' cannot occur when one sums operators in a Roe-type algebra.  The second is perhaps the main ingredient in our results.  

Section \ref{fadsec} then proves Theorem \ref{main} (and some related results).  Section \ref{asec} proves the purely algebraic statement from Theorem \ref{main1} and related results.   Section \ref{cpsec} provides the extra technicalities needed to prove the Morita equivalence results from Theorem \ref{main1} and Corollary \ref{groupc}.

Finally, Appendix \ref{cat} defines categories of spaces in which the morphisms are (equivalence classes of) coarse equivalences and (equivalence classes of) spatially implemented isomorphisms of Roe algebras, and shows that these two categories are isomorphic.

\section{Definitions and notation}\label{defsec}

All Hilbert spaces considered in this piece are separable (possibly finite dimensional) and complex, and inner products are linear in the second variable, conjugate linear in the first.  We write $\mathcal{B}(\mathcal{H})$ (respectively, $\mathcal{K}(\mathcal{H})$) for the $C^*$-algebra of bounded (respectively, compact) operators on $\mathcal{H}$. 

If $X$ is a set and $\mathcal{H}$ a Hilbert space, $l^2(X,\mathcal{H})$ denotes the Hilbert space of square summable functions on $X$ with values in $\mathcal{H}$.  $l^2(X)$ denotes $l^2(X,\C)$, and the standard orthonormal basis of this space is denoted $\{\delta_x\}_{x\in X}$.  If $l^2(X)\otimes\h$ is the (completed) Hilbert space tensor product, then there is a natural isomorphism 
$$
l^2(X)\otimes\mathcal{H}\cong l^2(X,\mathcal{H})
$$
defined by sending the elementary tensor $\delta_x\otimes v$ to the function from $X$ to $\mathcal{H}$ defined by
$$
y\mapsto \left\{\begin{array}{ll} v & y=x \\ 0 & \text{otherwise} \end{array}\right.;
$$
 we will often make use of this isomorphism, and the corresponding natural representation of the spatial tensor product $\mathcal{B}(l^2(X))\otimes\mathcal{B}(\mathcal{H})$ on $l^2(X,\mathcal{H})$, without further comment.

If $A$ is a subset of $X$ we write $\chi_A$ for the characteristic function of $A$ considered as a multiplication operator on $l^2(X,\mathcal{H})$.  For $x,y\in X$ we write $e_{xy}$ for the `matrix unit' in $\mathcal{B}(l^2(X,\mathcal{H}))$ defined on elementary tensors by
$$
e_{xy}:\delta_z\otimes v\mapsto \langle \delta_y,\delta_z\rangle \delta_x\otimes v.
$$ 
Similarly, if $x,y\in X$ and $v,w\in \mathcal{H}$, we write $e_{(x,v),(y,w)}\in\mathcal{B}(l^2(X,\mathcal{H}))$ for the rank one `matrix unit' defined by
\begin{equation}\label{mu}
e_{(x,v),(y,w)}:\delta_z\otimes u \mapsto \langle \delta_y\otimes w,\delta_z\otimes u\rangle\delta_x\otimes v.
\end{equation}
If $T$ is a bounded operator from $l^2(Y,\h)$ to $l^2(X,\h)$, $x\in X$ and $y\in Y$, we write 
$$
T_{xy}:=e_{xx}Te_{yy},
$$
for the `$(x,y)^\text{th}$ matrix coefficient of $T$', which naturally identifies with a bounded operator on $\mathcal{H}$.  

\begin{roealg2}\label{roealg2}
Let $X$ be a space and $\h$ a separable (possibly finite dimensional) Hilbert space.  Let  
$$T=(T_{xy})_{x,y\in X}$$
be a bounded operator on $l^2(X,\h)$.  The operator $T$ is said to have \emph{finite propagation} if there exists $S>0$ such that if $d(x,y)\geq S$ then $T_{xy}=0$.  The operator $T$ is said to be \emph{locally compact} if each $T_{xy}$ is a compact operator on $\h$.

A $*$-subalgebra of $\mathcal{B}(l^2(X,\h))$ is called an \emph{algebraic Roe-type algebra} if
\begin{itemize}
\item it consists only of locally compact, finite propagation operators;
\item it contains any finite propagation operator $T$ such that there exists $N\in\N$ such that the rank of $T_{xy}$ is at most $N$ for all $x,y\in X$.
\end{itemize}
We denote algebraic Roe-type algebras by $A[X]$\footnote{Note that there are many non-isomorphic algebraic Roe-type algebras associated to $X$, so this notation is slightly abusive.}.

A $C^*$-subalgebra of $\mathcal{B}(l^2(X,\h))$ is called a \emph{Roe-type algebra} if it is the norm closure of some algebraic Roe-type algebra.  We denote Roe-type algebras by $A^*(X)$.
\end{roealg2}

The following gives some examples of Roe-type algebras, and one related example that does not quite satisfy the conditions in Definition \ref{roealg2} above; all have applications in coarse geometry and index theory.

\begin{roeex}\label{roeex}
Let $X$ be a space and $\h$ be a separable Hilbert space.  Consider the following $*$-algebras of operators on $l^2(X,\h)$.
\begin{itemize}
\item say $\h=\C$ and let $\C_u[X]$ consist of all finite propagation operators on $l^2(X,\mathcal{H})=l^2(X)$;
\item say $\h$ is infinite dimensional and let $\C_s[X]$ consists of all finite propagation operators $T$ on $l^2(X,\h)$ such that there exists a finite dimensional subspace $\h_T\subseteq \h$ such that for all $x,y\in X$, $T_{xy}\in\mathcal{B}(\h_T)$;
\item say $\h$ is infinite dimensional and let $U\C[X]$ consist of all finite propagation operators $T$ such that there exists $N\in\N$ such that for all $x,y\in X$, $T_{xy}$ is an operator of rank at most $N$;
\item say $\h$ is infinite dimensional and let $\C[X]$ consist of finite propagation operators $T$ such that for all $x,y\in X$, $T_{xy}$ is in $\mathcal{K}(\h)$.
\end{itemize}
The operator norm completions of these four $*$-algebras are denoted $C^*_u(X)$, $C_s^*(X)$, $UC^*(X)$ and $C^*(X)$ respectively, and are called the \emph{uniform Roe algebra of $X$}, \emph{stable uniform Roe algebra of $X$}, \emph{uniform algebra of $X$}, and \emph{Roe algebra of $X$}, respectively.   

The $*$-algebras $\C_u[X]$, $U\C[X]$ and $\C[X]$ are all algebraic Roe-type algebras, and the corresponding completions are all Roe-type algebras.

The algebras $\C_s[X]$ are not algebraic Roe-type algebras, but they are close enough that the techniques in this paper will still apply with some minor elaborations.
\end{roeex}

\begin{roerem}\label{roerem}
\begin{enumerate}
\item It is not hard to see that $C_s^*(X)$ is canonically isomorphic to $C^*_u(X)\otimes\mathcal{K}(\h)$, whence its name.
\item Clearly one has inclusions
$$
C^*_u(X)\subseteq C^*_s(X)\subseteq UC^*(X)\subseteq C^*(X),
$$
the last two of which are canonical, given the choice of an infinite dimensional Hilbert space.  In general, none of these inclusions are equalities, and moreover none of these four $C^*$-algebras associated to $X$ is abstractly isomorphic to any of the others.
\item It will be important in what follows that any algebraic Roe-type algebra associated to a space $X$ contains all the rank one operators $e_{(x,v),(y,w)}$, where $x,y\in X$ and $v,w\in\h$.  It follows in particular that any Roe-type algebra  associated to $X$ contains the compact operators $\mathcal{K}(l^2(X,\h))$ as a $C^*$-subalgebra.  The analogous facts are also true for $\C_s[X]$ and $C^*_s(X)$. 
\item The $C^*$-algebras $C_s^*(X)$, $UC^*(X)$ and $C^*(X)$ are all coarse invariants, in the sense that if $X$ and $Y$ are coarsely equivalent spaces, then $C_s^*(X)\cong C^*_s(Y)$, $UC^*(X)\cong UC^*(Y)$ and $C^*(X)\cong C^*(Y)$.  Similar statements hold on the purely algebraic level, but not for $C^*_u(X)$ as one can see by considering finite spaces: in this case coarse equivalence of $X$ and $Y$ implies Morita equivalence of $C^*_u(X)$ and $C^*_u(Y)$  \cite[Theorem 4]{Brodzki:2007mi}.
\item The Roe-type algebras $C^*_s(X)$, $UC^*(X)$ and $C^*(X)$ can be made sense of for metric spaces that are coarsely equivalent to bounded geometry (discrete) spaces, rather than metric spaces that actually have these properties themselves.  The analogues of Theorem \ref{main1} still hold in this context.
\end{enumerate}
\end{roerem}

\section{Analytic Preliminaries}\label{apsec}

The following lemma was pointed out to us by Christian Voigt.  It says roughly that any $*$-isomorphism of (algebraic) Roe-type algebras is spatially implemented; the statement is slightly more technical than this to allow it to apply to both the algebraic and $C^*$-algebraic cases directly.  

\begin{spatiallem}\label{spatiallem}
Let $X$ and $Y$ be spaces.  Let $A[X]$ denote either an algebraic Roe-type algebra or $\C_s[X]$, and $A^*(Y)$ denote either a Roe type algebra or $\C_s[Y]$.  Assume these algebras are represented on Hilbert spaces $l^2(X,\h_X)$ and $l^2(Y,\h_Y)$ respectively.  Let 
$$
\phi:A[X]\to A^*(Y)
$$
be a $*$-homomorphism that is a $*$-isomorphism onto a dense $*$-subalgebra of $A^*(Y)$. 

Then there exists a unitary isomorphism $U:l^2(X,\h_X)\to l^2(Y,\h_Y)$ spatially implementing $\phi$ in the sense that
$$
\phi(T)=UTU^*
$$
for all $T\in A[X]$.  

In particular, $\phi$ extends to a $*$-isomorphism from $A^*(X)$ to $A^*(Y)$ that is continuous for the strong topology.
\end{spatiallem}

\begin{proof}
For simplicity of notation in what follows, write $\K_X$, $\K_Y$ for $\K(l^2(X,\h_X))$, $\K(l^2(Y,\h_Y))$ respectively. 

Let $\{v_n\}_{n\in\N}$ be an orthonormal basis of $\h_X$, and consider the collection $\{e_{(x,v_n),(z,v_m)}\}_{x,z\in X,n,m\in\N}$ of matrix units in $\K_X$, which are also in $A[X]$ (cf.\ part 3 of Remark \ref{roerem}).  Let $I$ denote the $*$-algebra spanned by these matrix units, which has closure $\K_X$; moreover, it follows from the fact that $I$ is an increasing union of matrix algebras that $\K_X$ is the unique $C^*$-completion of $I$.  Hence  the norm $\phi(I)$ inherits from $A^*(Y)$ is the same as the $\K_X$-norm, and thus $\phi$ extends (uniquely) to a $*$-homomorphism
$$
\phi:A[X]+\mathcal{K}_X\to A^*(Y)
$$
which is of course still a $*$-isomorphism onto a dense $*$-subalgebra.

Now, it follows from density of $\phi(A[X])$ in $A^*(Y)$ that $\phi(\K_X)$ is an ideal in $A^*(Y)$, and therefore contains $\K_Y$ as an ideal; however, $\phi(\K_X)$ is simple (as $\K_X$ is), whence $\phi(\K_X)=\K_Y$, i.e.\ $\phi$ restricts to a $*$-isomorphism from $\K_X$ to $\K_Y$.  Any such $*$-isomorphism is spatially implemented (see for example \cite[Corollary 4.1.8]{Dixmier:1977vl}), whence there exists a unitary isomorphism $U:l^2(X,\h_X)\to l^2(Y,\h_Y)$ such that
\begin{equation}\label{spaimp}
\phi(K)=UKU^* ~\text{ for all }~ K\in\mathcal{K}_X.
\end{equation}

Note now that for any $T\in A[X]$ and any $x,z\in X$, $n,m\in \N$ we have that 
\begin{align*}
e_{(x,v_n),(x,v_n)}U^*\phi(T)Ue_{(z,v_m),(z,v_m)} & = U^*Ue_{(x,v_n),(x,v_n)}U^*\phi(T)Ue_{(z,v_m),(z,v_m)}U^*U \\ &=U^*\phi(e_{(x,v_n),(x,v_n)})\phi(T)\phi(e_{(z,v_m),(z,v_m)})U \\ & =U^*\phi(e_{(x,v_n),(x,v_n)}Te_{(z,v_m),(z,v_m)})U \\& =e_{(x,v_n),(x,v_n)}Te_{(z,v_m),(z,v_m)},
\end{align*}
where the last equality follows from line \eqref{spaimp}.  This says that $U^*\phi(T)U$ and $T$ have the same matrix coefficients, whence 
$$
\phi(T)=UTU^*
$$
for all $T\in A[X]$.  It follows that $\phi$ is continuous for both the norm and strong topologies, and thus extends to a $*$-isomorphism from $A^*(X)$ to $A^*(Y)$.  
\end{proof}

The next lemma, and the similar Lemma \ref{jlem2}, are the most important tools in proving our main results.

\begin{jlem}\label{jlem}
Let $X$ be a space, and $A^*(X)$ denote either a Roe-type algebra associated to $X$ or $C^*_s(X)$.

Let $(T_n)_{n\in \N}$ be a sequence of mutually orthogonal operators in $A^*(X)$, and assume that for any $\mathcal{N}\subseteq \N$, the sum $\sum_{n\in \mathcal{N}}T_n$ converges strongly to some operator in $A^*(X)$.

Then for any $\epsilon>0$ there exists $S\geq 0$ such that for any subsets $A$, $B$ of $X$ with  $d(A,B)>S$, and any $n$, we have that
$$
\|\chi_AT_n\chi_B\|<\epsilon.
$$
\end{jlem}

\begin{proof}
Choose a sequence $(p_k)$ of finite rank projections in $\mathcal{H}$ that converge strongly to the identity, and let 
$$
P_k=1\otimes p_k\in l^2(X)\otimes \h.
$$ 
Choose also sequences of nested finite subsets $(A_k)$, $(B_k)$ of $A$, $B$ respectively such that $A=\cup A_k$ and $B=\cup B_k$.  As for any $n$ we have that
$$
(\chi_{A_k} P_k)T_n(\chi_{B_k} P_k)\stackrel{s.o.t.}{\longrightarrow}\chi_AT_n\chi_B~~ \text{ as }~~ k\to\infty
$$
(this uses that the sequences $(\chi_{A_k}P_k)$ and $(\chi_{B_k} P_k)$ are uniformly bounded), and as limits in the strong operator topology do not increase norms, it suffices to prove the following statement: `for any $\epsilon>0$ there exists $S\geq 0$ such that for any finite subsets $A$, $B$ of $X$ such that $d(A,B)>S$, all $k$ and any $n$, we have that
$$
\|\chi_A P_kT_n\chi_BP_k\|<\epsilon.\text{'}
$$
We will prove this statement.

Assume for contradiction that there is $\epsilon>0$, such that for any $S\geq 0$ there are $S$-separated finite subsets $A_S,B_S\subseteq X$, a projection $P_{k_S}=1\otimes p_{k_S}$ and $i_S$, such that 
$$\|\chi_{A_S}P_{k_S}T_{i_S}\chi_{B_S}P_{k_S}\|\geq \epsilon.$$
Choose any sequence $(S_n)$ tending to infinity, and for simplicity rename $A_{S_n}$, $B_{S_n}$, $P_{k_{S_n}}$ and $i_{S_n}$ as $A_n$, $B_n$, $P_n$ and $n$ respectively.

We will show that there are two subsets $\mathcal{N}, \mathcal{M}\subseteq \N$ such that:
\begin{itemize}
\item  $\lim_{n\in \mathcal{N}}S_n=\infty$ (in other words, $\mathcal{N}$ is infinite); and
\item for all $n\in \mathcal{N}$, 
$$\Big\|\chi_{A_n}P_n\Big(\sum_{m\in \mathcal{M}} T_m\Big)\chi_{B_n}P_n\Big\|\geq\frac\epsilon4.$$
\end{itemize} 
This contradicts the fact that $\sum_{m\in \mathcal{M}} T_m$ is in $A^*(X)$.

We define $\mathcal{N}$ and $\mathcal{M}$ inductively; start with $\mathcal{N}=\mathcal{M}=\emptyset$, and a parameter $k=1$. For the inductive step, we assume that we have `updated' $\mathcal{N}$, $\mathcal{M}$ and $k$.
We declare the updated $\mathcal{N}$ to be $\mathcal{N}\cup\{k\}$. Now, by assumption, 
$$\|\chi_{A_k}P_kT_k\chi_{B_k}P_k\|\geq\epsilon,$$
whence either $\mathcal{M}':=\mathcal{M}$ or $\mathcal{M}':=\mathcal{M}\cup\{k\}$ satisfies 
$$\Big\|\chi_{A_k}P_k\Big(\sum_{m\in \mathcal{M}'}T_m\Big)\chi_{B_k}P_k\Big\|\geq\frac\epsilon2.$$ 
We rename $\mathcal{M}'$ to $\mathcal{M}$. As the operators $T_m$ are mutually orthogonal and the sum $\sum_{m>k}T_m$ converges strongly, the sum
\begin{equation}\label{rr}
\sum_{m> k}\chi_{A_k}P_kT_{\sigma(m)}\chi_{B_k}P_k
\end{equation}
converges strongly for any permutation $\sigma$ of $\{m\in\N~|~m>k\}$; the Riemann rearrangement theorem for the finite dimensional vector space
$$
\mathcal{B}(\chi_{B_k}P_k\cdot(l^2(X,\mathcal{H}))~,~\chi_{A_k}P_k\cdot(l^2(X,\mathcal{H})))
$$
thus implies that the sum in line \eqref{rr} above converges absolutely; in particular, there exists $k'> k$ such that 
$$\sum_{m\geq k'}\|\chi_{A_k}P_kT_m\chi_{B_k}P_k\|\leq \frac\epsilon4.$$ 
Hence  if $\mathcal{M}'$ is the result of adding any number of elements to $\mathcal{M}$, all of which are larger than $k'$, then
 $$\Big\|\chi_{A_k}P_k\Big(\sum_{m\in \mathcal{M'}}T_m\Big)\chi_{B_k}P_k\Big\|\geq \frac\epsilon2-\frac\epsilon4=\frac\epsilon4.$$ 
We declare the new $k$ to be $k'$ and repeat the induction. It is clear that this process yields the required subsets $\mathcal{N}$ and $\mathcal{M}$.
\end{proof}

We will also need the following slightly different version of the above lemma; as the proof is essentially the same, we omit it.

\begin{jlem2}\label{jlem2}
Let $X$ be a space, and let $A[X]$ denote either an algebraic Roe-type algebra associated to $X$, or $C^*_s(X)$.

Let $(T_n)_{n\in \N}$ be a sequence of mutually orthogonal operators in $A[X]$, and assume that for any $\mathcal{N}\subseteq \N$, the sum $\sum_{n\in \mathcal{N}}T_n$ converges strongly to some operator in $A[X]$.

Then there exists $S\geq 0$ such that for any subsets $A$, $B$ of $X$ with  $d(A,B)>S$, and any $n$, we have that
$$
\chi_AT_n\chi_B=0. \eqno \qed
$$
\end{jlem2}

\section{$C^*$-algebraic isomorphisms and coarse equivalences}\label{fadsec}

In this section, we will prove the theorem below.

\begin{fadthe}\label{fadthe}
Say $X$ and $Y$ are spaces with property A, and $A^*(X)$ and $A^*(Y)$ are Roe-type algebras associated to $X$ and $Y$ respectively.

If there exists a $*$-isomorphism $\phi:A^*(X)\to A^*(Y)$, then $X$ and $Y$ are coarsely equivalent.
\end{fadthe}

We will not use property A directly here, but rather the property in the following definition, which was introduced by Chen--Tessera--Wang--Yu \cite[Section 3]{Chen:2008so}.  

\begin{msp}\label{msp}
Let X be a space. X is said to have the \emph{metric sparsification property} if there exists a constant $1 \geq\kappa> 0$ such that for all $S\geq0$ there exists $D\geq0$ such that for every Borel probability measure $\mu$ on $X$ there exists a subset $\Omega\subseteq X$ equipped with a decomposition $\Omega = \sqcup_{i\in I} \Omega_i$ such that:
\begin{itemize}
\item $\mu(\Omega)\geq \kappa$;
\item $\text{diam}(\Omega_i)\leq D$ for all $i$;
\item $d(\Omega_i,\Omega_j)>S$ for all $i \neq j$. 
\end{itemize}
\end{msp}

\begin{amsp}\label{amsp}
Let $X$ be a space.  Then $X$ has property A if and only if it has the metric sparsification property.
\end{amsp}

\begin{proof}
This follows from \cite[Proposition 3.2 and Theorem 3.8]{Brodzki:2011fk} and \cite[Theorem 4.1]{Sako:2012fk}.
\end{proof}

For the remainder of this section, fix $X, Y$ and $\phi:A^*(X)\to A^*(Y)$ as in the statement of Theorem \ref{fadthe}; using Theorem \ref{amsp}, $X$ and $Y$ have the metric sparsification property.  Note that Lemma \ref{spatiallem} implies that there exists a unitary isomorphism
$
U:l^2(X,\h_X)\to l^2(Y,\h_Y)
$
spatially implementing $\phi$;  we also fix $U$ throughout.  Finally, for the remainder of this section, fix unit vectors $v_0\in\h_X$, $w_0\in \h_Y$.

The following very simple lemma is isolated for ease of reference; see line \eqref{mu} above for notation.

\begin{cloudlem}\label{cloudlem}
For any $x_1,x_2\in X$, $v_1,v_2\in\h_X$, $y_1,y_2\in Y$, and $w_1,w_2\in\h_Y$ we have the formula
\begin{align*}
e_{(y_1,w_1)}(\phi(e_{(x_1,v_1),(x_2,v_2)}))&e_{(y_2,w_2)} \\ &=\langle \delta_{y_1}\otimes w_1,U(\delta_{x_1}\otimes v_1)\rangle\langle U(\delta_{x_2}\otimes v_2),\delta_{y_2}\otimes w_2\rangle.
\end{align*}
\end{cloudlem}

\begin{proof}
Computing,
\begin{align*}
e_{(y_1,w_1)}(\phi(e_{(x_1,v_1),(x_2,v_2)}))&e_{(y_2,w_2)} \\
 & =\langle\delta_{y_1}\otimes w_1,\phi(e_{(x_1,v_1),(x_2,v_2)})(\delta_{y_2} \otimes w_2)\rangle \\ 
 & =\langle U^*(\delta_{y_1}\otimes w_1),e_{(x_1,v_1),(x_2,v_2)}U^*(\delta_{y_2}\otimes w_2)\rangle \\ 
 & =\langle U^*(\delta_{y_1}\otimes w_1),\langle \delta_{x_2}\otimes v_2,U^*(\delta_{y_2}\otimes w_2)\rangle\delta_{x_1}\otimes v_1\rangle \\ 
 & =\langle \delta_{x_2}\otimes v_2,U^*(\delta_{y_2}\otimes w_2)\rangle \langle U^*(\delta_{y_1}\otimes w_1), \delta_{x_1}\otimes v_1\rangle \\
 & =\langle U(\delta_{x_2}\otimes v_2),\delta_{y_2}\otimes w_2\rangle \langle \delta_{y_1}\otimes w_1, U(\delta_{x_1}\otimes v_1)\rangle
\end{align*} 
as required.
\end{proof}

\begin{cloudbound1}\label{cloudbound1}
\begin{enumerate}
\item Let $\delta>0$, and for each $x\in X$, fix a finite dimensional subspace $E_x$ of $\h_X$.  Then there exists $S\geq 0$ such that if $x\in X$,  unit vectors $v_1,v_2\in E_x$, $y_1,y_2\in Y$ and unit vectors $w_1,w_2\in \h_Y$ are such that
$$
|\langle U(\delta_x\otimes v_1),\delta_{y_1}\otimes w_1\rangle|\geq \delta~~\text{ and }~~|\langle U(\delta_x\otimes v_2),\delta_{y_2}\otimes w_2\rangle|\geq \delta,
$$
then $d(y_1,y_2)\leq S$.
\item Let $\delta>0$, and for each $x\in X$, fix a finite dimensional subspace $E_x$ of $\h_X$.  Then for all $R\geq 0$ there exists $S\geq0$ such that if such that if $x_1,x_2\in X$,  unit vectors $v_1\in E_{x_1}$, $v_2\in E_{x_2}$, $y_1,y_2\in Y$ and unit vectors $w_1,w_2\in \h_Y$ are such that
$$
|\langle U(\delta_{x_1}\otimes v_1),\delta_{y_1}\otimes w_1\rangle|\geq \delta~~\text{ and }~~|\langle U(\delta_{x_2}\otimes v_2),\delta_{y_2}\otimes w_2\rangle|\geq \delta,
$$
then $d(y_1,y_2)\leq S$.
\end{enumerate}
The same properties hold with the roles of $X$ and $Y$ reversed, and with $U$ replaced by $U^*$.
\end{cloudbound1}

\begin{proof}
The fact that the same properties hold with the roles of $X$ and $Y$ interchanged follows by symmetry.  Moreover, the first of these properties is a special case of the second.  It thus suffices to prove the second property in the form stated.

Assume for contradiction that this is false.  Then there exists $R\geq 0$ and sequences $(x_1^n)$, $(x^n_2)$, $(v_1^n)$, $(v_2^n)$, $(y^n_1)$, $(y^n_2)$, $(w_1^n)$, $(w_2^n)$ such that $d(x^n_1,x^n_2)\leq R$ for all $n$, so that $v_i^n\in E_{x_i^n}$ for all $n$ and $i=1,2$, so that 
\begin{equation}\label{lowbound}
|\langle U(\delta_{x_1^n}\otimes v_1^n),\delta_{y_1^n}\otimes w_1^n\rangle|\geq\delta~~ \text{ and }~~|\langle U(\delta_{x_2^n}\otimes v_2^n),\delta_{y_2^n}\otimes w_2^n\rangle|\geq \delta,
\end{equation}
and so that $d(y_1^n,y_2^n)\to \infty$ as $n\to\infty$.  Note that at least one of the sequences $(y_i^n)$, $i=1,2$ must have a subsequence tending to infinity in $Y$; say without loss of generality this is $(y_1^n)$ and passing to a subsequence, assume $(y_1^n)$ itself tends to infinity.   It follows that the sequence $(\delta_{y_1^n}\otimes w_1^n)$ of unit vectors in $l^2(Y)\otimes \mathcal{H}_Y$ tends weakly to zero; thus for the lower bounds in line \eqref{lowbound} to be possible, $(\delta_{x_1^n}\otimes v_1^n)$ must eventually leave any (norm) compact subset of $l^2(X)\otimes \h_X$;  finite dimensionality of each subspace $E_{x_1^n}$ thus forces $(x_1^n)$ to tend to infinity in $X$.  Passing to another subsequence, we may assume that $d(x_1^n,x_1^{n+1})>2R$ for all $n$; this and the fact that $d(x_1^n,x_2^n)\leq R$ for all $n$ implies in particular that all of the elements $x^n_1$, $x_2^n$ of $X$ are distinct.  

Now, distinctness of the elements $x_1^n$, $x_2^n$ implies that the operators $e_{(x_1^n,v_1^n),(x_2^n,v_2^n)}$ are mutually orthogonal and so for any $\mathcal{N}\subseteq \N$ the sum
$$
\sum_{n\in \mathcal{N}}e_{(x_1^n,v_1^n),(x_2^n,v_2^n)}
$$
converges strongly to a bounded operator on $l^2(X)\otimes\h_X$, that is moreover in $A^*(X)$\footnote{This need not be true in $C^*_s(X)$, which is the reason for treating that case separately.}.  Hence the same is true of the sums 
$$
\sum_{n\in \mathcal{N}}\phi(e_{(x_1^n,v_1^n),(x_2^n,v_2^n)}),
$$
in $A^*(Y)$, using strong continuity of $\phi$ (which follows from Lemma \ref{spatiallem}).
Lemma \ref{jlem} thus implies that there exists $S>0$ such that for any $A,B\subseteq Y$ with $d(A,B)>S$ we have that
$$
\|\chi_A\phi(e_{(x_1^n,v_1^n),(x_2^n,v_2^n)})\chi_B\|<\delta^2.
$$
Set $A_n=\{y^n_1\}$, $B_n=\{y^n_2\}$.  Then we have that for all $n$ so large that $d(y^n_1,y^n_2)>S$
\begin{align*}
\|\chi_{A_n}\phi(e_{(x_1^n,v_1^n),(x_2^n,v_2^n)}) & \chi_{B_n}\| \geq \|e_{(y_1^n,w_1^n)}\phi(e_{(x_1^n,v_1^n),(x_2^n,v_2^n)})e_{(y^n_2,w_2^n)}\| \\
& =|\langle \delta_{y_1^n}\otimes w_1^n,U(\delta_{x_1^n}\otimes v_1^n)\rangle \langle U(\delta_{x^n_2}\otimes v_2^n),\delta_{y^n_2}\otimes w^n_2\rangle|,
\end{align*}
where the last equality uses Lemma \ref{cloudlem}.  This is greater than  $\delta^2$ by assumption, however, which is a contradiction.
\end{proof}

The following analytic lemma will be our only use of the metric sparsification property.

\begin{fadlem2}\label{fadlem2}
There exists $c>0$ such that for each $x\in X$ there exists $f(x)\in Y$ and a unit vector $w_x\in \h_X$ such that 
$$
|\langle U(\delta_x\otimes v_0), \delta_{f(x)}\otimes w_x\rangle| \geq c.
$$
The same property holds with the roles of $X$ and $Y$ reversed, $U$ replaced by $U^*$ and $v_0$ replaced by $w_0$.
\end{fadlem2}

\begin{proof}
Let $\kappa$ be as in Definition \ref{msp}. Let $\epsilon=\kappa/4$ and $t=\kappa/5$, so in particular we have that
\begin{equation}\label{teps}
t+\epsilon<\frac{\kappa}{2}.
\end{equation}
Applying Lemma \ref{jlem} to the sum
$$
\sum_{x\in X}\phi(e_{(x,v_0),(x,v_0)}) \in A^*(X)
$$ 
(similarly to the proof of Lemma \ref{cloudbound1} above) implies that there exists $S>0$ so that for any $x\in X$ and $A,B\subseteq Y$ with $d(A,B)> S$ we have 
\begin{equation}\label{bound1}
\|\chi_A \phi(e_{(x,v_0),(x,v_0)})\chi_B\|<\epsilon.
\end{equation}
For the remainder of the proof, fix $x\in X$.  Define $\xi=U(\delta_x\otimes v_0)$, which we think of as a function from $Y$ to $\h$.  Note that $\chi_A\phi(e_{(x,v_0),(x,v_0)})\chi_B$ is equal to the rank one operator on $l^2(Y,\h_Y)$ defined by
$$
\eta\mapsto \langle \xi|_B ,\eta\rangle\xi|_A
$$
(where `$|_C$' denotes restriction to a subset $C\subseteq Y$), whence line \eqref{bound1} above implies that
\begin{equation}\label{fadinq}
\|\xi|_A\|\|\xi|_B\|<\epsilon.
\end{equation}

Now, using the metric sparsification property for $X$ applied to the Borel measure (in the current context, just a non-negative function with total mass one)
\begin{align*}
\mu:X & \to [0,1] \\
x & \mapsto \|\xi(x)\|^2 
\end{align*}
there exists a subset $\Omega=\sqcup_{i\in I} \Omega_i$ of $X$ and $D\geq 0$ such that $\text{diam}(\Omega_i)\leq D$ for all $i$, $d(\Omega_i,\Omega_j)> S$ for all $i\neq j$ and   
$$
\|\xi|_\Omega\|^2\geq\kappa.
$$

Assume for contradiction that $\|\xi|_{\Omega_i}\|^2<t$ for all $i$.  Then there exists a partition $I=I_1\sqcup I_2$ such that
$$
\sum_{i\in I_1}\|\xi|_{\Omega_i}\|^2\geq \frac{\kappa}{2}-t~~ \text{ and }~~\sum_{i\in I_2}\|\xi|_{\Omega_i}\|^2\geq\frac{\kappa}{2}-t. 
$$
Taking $A=\sqcup_{i\in I_1}\Omega_i$ and $B=\sqcup_{i\in I_2}\Omega_i$, line \eqref{fadinq} above implies that
$$
\frac{\kappa}{2}-t\leq\|\xi|_A\|\|\xi|_B\|<\epsilon,
$$ 
which contradicts line \eqref{teps} above.

Hence $\|\xi|_{\Omega_i}\|^2\geq t$ for some $i$.  Using bounded geometry and uniform boundedness of the sets $\Omega_i$, there is some $M\in\N$ (that depends only on $S$) such that $|\Omega_i|\leq M$ for all $i$.  Hence there is $y\in \Omega_i$ such that
$$
\|\xi(y)\|^2\geq t/M. 
$$
The lemma is then true for $f(x)$ equal to this $y$,  $w_x=\xi(y)/\|\xi(y)\|$, and $c=\sqrt{t/M}$.

The fact that the same result holds for $Y$ follows by symmetry.
\end{proof}

\begin{proof}[Proof of Theorem \ref{fadthe}]
Let $c>0$ have the property in Lemma \ref{fadlem2} both as stated, and with the roles of $X$ and $Y$ reversed.  Lemma \ref{fadlem2} implies that for each $x\in X$, there exists an element $f(x)$ of $Y$ and a unit vector $w_x\in \h_Y$ such that
$$
|\langle U(\delta_x\otimes v_0),\delta_{f(x)}\otimes w_x\rangle|\geq c;
$$
in particular, this defines a function $f:X\to Y$.  Similarly, for any $y\in Y$, there exists $g(y)\in X$ and a unit vector $v_y\in \h_X$ such that
$$
|\langle U^*(\delta_y\otimes w_0),\delta_{g(y)}\otimes v_y\rangle|\geq c,
$$
defining a function $g:Y\to X$.  Note that part 2 of Lemma \ref{cloudbound1} (with $E_x$ taken to be the span of $\{v_0\}$) implies that both $f:X\to Y$ and $g:Y\to X$ are uniformly expansive.  It remains to show that the compositions $f\circ g$ and $g\circ f$ are close to the identity.  

Fix for the moment $x\in X$ and assume for contradiction that the subset $g^{-1}(x)$ of $Y$ is infinite.  Define
$$
P=\sum_{y\in g^{-1}(x)}e_{(y,w_0),(y,w_0)},
$$  
which is an infinite rank projection in $A^*(Y)$.  Note that for any $y\in g^{-1}(x)$, we have that
\begin{equation}\label{lowbound}
|\langle U^*(\delta_y\otimes w_0),\delta_{x}\otimes v_y\rangle|\geq c.
\end{equation}
Let $P_x\in \B(l^2(X)\otimes \h_X)$ be the orthogonal projection onto $\text{span}\{\delta_x\}\otimes \h_X$.  Then the operator $P_xT$ is compact for any $T\in A^*(X)$ (as $T$ is a limit of locally compact finite propagation operators), whence $P_x(U^*PU)$ is compact.  On the other hand, for any $y\in g^{-1}(x)$, 
\begin{align*}
\|P_x(U^*PU)(U^*e_{(y,w_0),(y,w_0)}U)\| & =\|P_xU^*e_{(y,w_0),(y,w_0)}\| \\ & \geq |\langle U^*(\delta_y\otimes w_0),\delta_{x}\otimes v_y\rangle|\\ & \geq c
\end{align*}
using line \eqref{lowbound} above.  As the operators $U^*e_{(y,w_0),(y,w_0)}U$ tend $*$-strongly to zero as $y$ tends to infinity (in $g^{-1}(x)$), this contradicts compactness of $P_xU^*PU$, and so $g^{-1}(x)$ is finite. 

For each $x\in X$, let then 
$$
E_x=\text{span} \big(\{v_0\}\cup\{v_y~|~y\in g^{-1}(x)\}\big),
$$
which is a finite dimensional subspace of $\h_X$.  Note that for any $y\in Y$, the choice of $f$, $g$, $v_y$, and $w_{f(g(y))}$ implies that
$$
|\langle U(\delta_{g(y)}\otimes v_0), \delta_{f(g(y))}\otimes w_{g(y)}\rangle|\geq c ~~\text{ and }~~|\langle U^*(\delta_y\otimes w_0),\delta_{g(y)}\otimes v_y\rangle|\geq c,
$$
or in other words that 
$$
|\langle U(\delta_{g(y)}\otimes v_0), \delta_{f(g(y))}\otimes w_{g(y)}\rangle|\geq c ~~\text{ and }~~|\langle U(\delta_{g(y)}\otimes v_y),\delta_y\otimes w_0\rangle|\geq c.
$$
Hence Lemma \ref{cloudbound1}, part 1, implies the existence of some $S$ (independently of $y$) such that 
$$
d(y,f(g(y)))\leq S.
$$
This says that $f\circ g$ is close to the identity.  Similarly, $g\circ f$ is close to the identity, and we are done.
\end{proof}

\section{Algebraic isomorphisms and coarse equivalences}\label{asec}

In this section, we will prove the following theorem.

\begin{athe}\label{athe}
Say $X$ and $Y$ are spaces and $A[X]$ and $A[Y]$ are algebraic Roe-type algebras associated to $X$ and $Y$ respectively.

If there exists a $*$-isomorphism $\phi:A[X]\to A[Y]$, then $X$ and $Y$ are coarsely equivalent.
\end{athe}

It follows from Lemma \ref{spatiallem} that $\phi$ is spatially implemented by a unitary isomorphism $U:l^2(X,\h_X)\to l^2(Y,\h_Y)$; we fix $U$ for the rest of this section.  We also fix unit vectors $v_0\in\h_X$ and $w_0\in\h_Y$.

The following lemma is an analogue of Lemma \ref{cloudbound1}.  The proof is essentially the same (using Lemmas \ref{jlem2} and \ref{cloudlem}), so we omit it.

\begin{cloudbound2}\label{cloudbound2}
\begin{enumerate}
\item For each $x\in X$, fix a finite dimensional subspace $E_x$ of $\h_X$.  Then there exists $S\geq 0$ such that if $x\in X$,  unit vectors $v_1,v_2\in E_x$, $y_1,y_2\in Y$ and unit vectors $w_1,w_2\in \h_Y$ are such that
$$
|\langle U(\delta_x\otimes v_1),\delta_{y_1}\otimes w_1\rangle|\neq0~~\text{ and }~~|\langle U(\delta_x\otimes v_2),\delta_{y_2}\otimes w_2\rangle|\neq0,
$$
then $d(y_1,y_2)\leq S$.
\item For each $x\in X$, fix a finite dimensional subspace $E_x$ of $\h_X$.  Then for all $R\geq 0$ there exists $S\geq0$ such that if such that if $x_1,x_2\in X$,  unit vectors $v_1\in E_{x_1},v_2\in E_{x_2}$, $y_1,y_2\in Y$ and unit vectors $w_1,w_2\in \h_Y$ are such that
$$
|\langle U(\delta_{x_1}\otimes v_1),\delta_{y_1}\otimes w_1\rangle|\neq 0~~\text{ and }~~|\langle U(\delta_{x_2}\otimes v_2),\delta_{y_2}\otimes w_2\rangle|\neq 0,
$$
then $d(y_1,y_2)\leq S$.
\end{enumerate}
The same properties hold with the roles of $X$ and $Y$ reversed, and with $U$ replaced by $U^*$. \qed
\end{cloudbound2}

\begin{proof}[Proof of Theorem \ref{athe}]
Define a function $f:X\to Y$ by choosing $f(x)$ to be any element of $Y$ such that there exists a unit vector $w_x\in \h_X$ such that
$$
|\langle U(\delta_x\otimes v_0),\delta_{f(x)}\otimes w_x\rangle|\neq 0.
$$
Similarly, define a function $g:Y\to X$ by choosing $g(y)$ to be any element of $X$ such that there exists a unit vector $v_y\in \h_Y$ such that
$$
|\langle U^*(\delta_y\otimes w_0),\delta_{g(y)}\otimes v_y\rangle|\neq 0.
$$
The proof concludes exactly as that of Theorem \ref{fadthe} above, using Lemma \ref{cloudbound2} in place of Lemma \ref{cloudbound1}.
\end{proof}

\section{The stable uniform case and crossed products}\label{cpsec}

In this section we prove our results in the stable uniform case.  This is slightly more subtle than the other cases we consider, but essentially the same argument works with a few minor changes to given analogues of Theorems \ref{fadthe} and \ref{athe} as below. 

\begin{suthe}\label{suthe}
\begin{enumerate}
\item Let $X$ and $Y$ be spaces with property A.  If there is a $*$-isomorphism
$\phi:C^*_s(X)\to C^*_s(Y)$,
then $X$ and $Y$ are coarsely equivalent.
\item Let $X$ and $Y$ be spaces.  If there is a $*$-isomorphism $\phi:\C_s[X]\to \C_s[Y]$, 
then $X$ and $Y$ are coarsely equivalent.
\end{enumerate}
\end{suthe} 

The following two corollaries are our main motivation for putting in the extra effort needed to prove this theorem; the first is part of Theorem \ref{main1} and the second is Corollary \ref{groupc} from the introduction.

\begin{morcor}\label{morcor}
Let $X$ and $Y$ be spaces with property A.  Then $C^*_u(X)$ is Morita equivalent to $C^*_u(Y)$ if and only if $X$ is coarsely equivalent to $Y$.
\end{morcor}

\begin{proof}
If $X$ is coarsely equivalent to $Y$, then $C^*_u(X)$ is Morita equivalent to $C^*_u(Y)$ by \cite[Theorem 4]{Brodzki:2007mi}.  On the other hand, if $C^*_u(X)$ and $C^*_u(Y)$ are Morita equivalent, then $C_s^*(X)\cong C^*_s(Y)$ by \cite[Theorem 1.2]{Lawrence-G.-Brown:1977co}, so the result follows from part 1 of Theorem \ref{suthe} above.
\end{proof}

\begin{groupc2}
Say $\Gamma$ and $\Lambda$ are finitely generated $C^*$-exact groups. Then the following are equivalent:
\begin{itemize}
\item $\Gamma$ and $\Lambda$ are quasi-isometric;
\item $l^\infty(\Gamma)\rtimes_r\Gamma$ and $l^\infty(\Lambda)\rtimes_r\Lambda$ are Morita equivalent.
\end{itemize}
If moreover $\Gamma$ and $\Lambda$ are non-amenable, then these are also equivalent to:
\begin{itemize}
\item $l^\infty(\Gamma)\rtimes_r\Gamma$ and $l^\infty(\Lambda)\rtimes_r\Lambda$ are $*$-isomorphic.
\end{itemize}
\end{groupc2}

\begin{proof}
For $\Gamma$ a (finitely generated) discrete group, the crossed product $l^\infty(\Gamma)\rtimes_r\Gamma$ is isomorphic to $C^*_u(X)$, where $X$ is the space defined by fixing any word metric on $\Gamma$ - this observation is due to Higson and Yu, see for example \cite[Proposition 5.1.3]{Brown:2008qy}.  Moreover, $C^*$-exactness is equivalent to property A by a result of Guentner--Kaminker \cite{Guentner:2022hc} and Ozawa \cite{Ozawa:2000th}.   

Corollary \ref{morcor} then implies that $\Gamma$ and $\Lambda$ are coarsely equivalent if and only if $l^\infty(\Gamma)\rtimes_r\Gamma$ and $l^\infty(\Lambda)\rtimes_r\Lambda$ are Morita equivalent.  Moreover, for finitely generated groups, coarse equivalence is well-know (and easily seen) to be the same as quasi-isometry.  

Finally, in the non-amenable case it follows from a result of Whyte \cite[Theorem 1.1]{Whyte:1999uq} that if $\Gamma$ and $\Lambda$ are quasi-isometric, then they are bi-Lipschitz equivalent.  This is easily seen to imply that $l^\infty(\Gamma)\rtimes_r\Gamma$ and $l^\infty(\Lambda)\rtimes_r\Lambda$ are $*$-isomorphic.
\end{proof}

We will concentrate on part 1 of Theorem \ref{suthe}: part 2 follows from the same ingredients, having made essentially the same changes we made when passing from Section \ref{fadsec} to Section \ref{asec} above.  As in Section \ref{fadsec}, fix $X$ and $Y$ with property A and a $*$-isomorphism $\phi$ as in the statement of Theorem \ref{suthe}, part 1, and let $U:l^2(X,\h_X)\to l^2(Y,\h_Y)$ be a unitary isomorphism spatially implementing $\phi$ as in Lemma \ref{spatiallem}.  Finally, fix unit vectors $v_0\in \h_X$ and $w_0\in\h_Y$.

The only really new ingredient we need in the proof of Theorem \ref{suthe} is the following lemma.

\begin{sulem}\label{sulem}
Say $c>0$ is such that for each $x\in X$ there exists $f(x)\in Y$ and a unit vector $w_x\in\h_Y$ such that
$$
|\langle U(\delta_x\otimes v_0),\delta_{f(x)}\otimes w_x\rangle|\geq c.
$$
Then there exists a finite rank projection $p\in\mathcal{B}(\h_X)$ such that
$$
|\langle U(\delta_x\otimes v_0),\delta_{f(x)}\otimes pw_x\rangle|\geq c/2
$$
for all $x\in X$.  

The same statement holds with the roles of $X$ and $Y$ reversed, $U$ replaced by $U^*$ and $v_0$ replaced by $w_0$.
\end{sulem}

\begin{proof}
We claim first that there exists a finite rank projection $p\in \mathcal{B}(\h_X)$ such that 
\begin{equation}\label{uniinq}
\|Ue_{(x,v_0),(x,v_0)}U^*(1-1\otimes p)\|\leq c/2
\end{equation}
for all $x\in X$.  Indeed, say otherwise, in which case there exists an increasing sequence of finite rank projections $(p_n)$ in $\mathcal{B}(\h_X)$ that converge strongly to the identity, and a sequence $(x_n)$ of elements of $X$ (which we may assume distinct) such that 
$$
\|Ue_{(x_n,v_0),(x_n,v_0)}U^*(1-1\otimes p_n)\|\geq c/2
$$  
for all $n$.  Moreover, there then exist finite rank projections $q_n\in\mathcal{B}(l^2(X,\h_X))$ such that
$$
\|q_nUe_{(x_n,v_0),(x_n,v_0)}U^*(1-1\otimes p_n)q_n\|\geq c/4
$$
for all $n$.  Using finite dimensionality of the projections $q_n$, an argument based on the Riemann rearrangement theorem as in the proof of Lemma \ref{jlem} shows that there exist subsets $\mathcal{M},\mathcal{N}\subseteq\N$ such that $\mathcal{N}$ is infinite, and such that
\begin{align*}
 \Big\|q_nU\Big(\sum_{m\in\mathcal{M}}e_{(x_m,v_0),(x_m,v_0)}\Big)U^*(1-1\otimes p_n)q_n\Big\|  \geq c/16
\end{align*}
for all $n\in\mathcal{N}$.  This, however, implies that
$$
\Big\|U\Big(\sum_{m\in\mathcal{M}} e_{(x_m,v_0),(x_m,v_0)}\Big)U^*(1-1\otimes p_n)\Big\| \geq c/16
$$
for all $n$, which contradicts that
$$
U\Big(\sum_{m\in\mathcal{M}}e_{(x_m,v_0),(x_m,v_0)}\Big)U^*=\phi\Big(\sum_{m\in\mathcal{M}}e_{(x_m,v_0),(x_m,v_0)}\Big)
$$
is an element of $C_s^*(Y)$.  This establishes the claim in line \eqref{uniinq}.

Note then that for any $x\in X$,
\begin{align*}
|\langle U(\delta_x\otimes v_0) & ,\delta_{f(x)}\otimes pw_x\rangle| =\|Ue_{(x,v_0),(x,v_0)}U^*(\delta_{f(x)}\otimes pw_x)\| \\
& \geq \|Ue_{(x,v_0),(x,v_0)}U^*(\delta_{f(x)}\otimes w_x)\|-\|Ue_{(x,v_0),(x,v_0)}U^*(1-1\otimes p)\| \\
& \geq c-c/2=c/2.
\end{align*}
The last statement is clear from symmetry.
\end{proof}

The next lemma is a close analogue of Lemma \ref{cloudbound1}: the only change is that we fix a finite dimensional subspace $E_X$ of $\h_X$ rather than allowing a family of finite dimensional subspaces $\{E_x\}_{x\in X}$.  Having made this change, the proof of Lemma \ref{cloudbound1} goes through essentially verbatim.

\begin{sucloud}\label{sucloud}
Fix a finite dimensional subspace $E_X$ of $\h_X$.
\begin{enumerate}
\item Let $\delta>0$.  Then there exists $S\geq 0$ such that if $x\in X$,  unit vectors $v_1,v_2\in E_X$, $y_1,y_2\in Y$ and unit vectors $w_1,w_2\in \h_Y$ are such that
$$
|\langle U(\delta_x\otimes v_1),\delta_{y_1}\otimes w_1\rangle|\geq \delta~~\text{ and }~~|\langle U(\delta_x\otimes v_2),\delta_{y_2}\otimes w_2\rangle|\geq \delta,
$$
then $d(y_1,y_2)\leq S$.
\item Let $\delta>0$.  Then for all $R\geq 0$ there exists $S\geq0$ such that if such that if $x_1,x_2\in X$,  unit vectors $v_1,v_2\in E_X$, $y_1,y_2\in Y$ and unit vectors $w_1,w_2\in \h_Y$ are such that
$$
|\langle U(\delta_{x_1}\otimes v_1),\delta_{y_1}\otimes w_1\rangle|\geq \delta~~\text{ and }~~|\langle U(\delta_{x_2}\otimes v_2),\delta_{y_2}\otimes w_2\rangle|\geq \delta,
$$
then $d(y_1,y_2)\leq S$.
\end{enumerate}
The same properties hold with the roles of $X$ and $Y$ reversed, and with $U$ replaced by $U^*$. \qed
\end{sucloud}

\begin{proof}[Proof of Theorem \ref{suthe}, part 1]

Just as in the proof of Theorem \ref{fadthe}, let $c>0$ have the property in Lemma \ref{fadlem2} (which still holds in this context, with the same proof) for both $X$ and $Y$.   Lemma \ref{fadlem2} implies that for each $x\in X$, there exists an element $f(x)$ of $Y$ and a unit vector $w_x\in \h_Y$ such that
$$
|\langle U(\delta_x\otimes v_0),\delta_{f(x)}\otimes w_x\rangle|\geq c;
$$
in particular, this defines a function $f:X\to Y$.  Similarly, for any $y\in Y$, there exists $g(y)\in X$ and a unit vector $v_y\in \h_X$ such that
$$
|\langle U^*(\delta_y\otimes w_0),\delta_{g(y)}\otimes v_y\rangle|\geq c,
$$
defining a function $g:Y\to X$.  Note that part 2 of Lemma \ref{sulem} (with $E_X$ taken to be the span of $\{v_0\}$) implies that $f:X\to Y$ and $g:Y\to X$ are uniformly expansive.  It remains to show that the compositions $f\circ g$ and $g\circ f$ are close to the identity.  

Using Lemma \ref{sulem} there exist finite rank projections $p_X\in \mathcal{B}(\h_X)$ and $p_Y\in \mathcal{B}(\h_Y)$ such that
$$
|\langle U(\delta_x\otimes v_0),\delta_{f(x)}\otimes p_Yw_x\rangle|\geq c/2
$$
for all $x\in X$ and  
$$
|\langle U^*(\delta_y\otimes w_0),\delta_{g(y)}\otimes p_X v_y\rangle|\geq c/2
$$
for all $y\in Y$.  Set $E_X=p_X\h_X$ and $E_Y=p_Y\h_Y$.  The proof can be completed just as in the case of Theorem \ref{fadthe}, using Lemma \ref{sucloud} in place of Lemma \ref{cloudbound1}.
\end{proof}

\appendix

\section{Categorical interpretation}\label{cat}

The aim of this appendix is to provide a slightly more conceptual framework for part 3 of Theorem \ref{main}.  For simplicity, we work throughout with the algebraic Roe algebras $\C[X]$, rather than one of the variants.  Throughout, we fix a separable, infinite dimensional Hilbert space $\h$, and assume that the algebraic Roe algebra of any space $X$ is  defined using $l^2(X,\h)$.

We show that coarse equivalences, and spatially implemented isomorphisms of algebraic Roe algebras are `essentially the same thing'.   More precisely, we define two categories of metric spaces: the morphisms in the first are given by closeness classes of coarse maps, and those in the second by `closeness' classes of spatially implemented isomorphisms of Roe algebras; we then show that these two categories are isomorphic.   

It is an unfortunate deficiency of our methods that all morphisms in our categories are isomorphisms: it is not currently clear to us if something more general is possible. Note also that the only novelty in this section is part 1 of Lemma \ref{covlem} (which follows immediately from our techniques in the rest of the piece); the remaining material comes from ideas used by Higson--Roe--Yu in their $K$-theoretic analyses of Roe algebras (see for example \cite[Section 4]{Higson:1993th}, which is probably the original reference along these lines.).

\begin{corcat}\label{corcat}
Define a category $\mathcal{C}_0$ by setting the objects to be uniformly discrete, bounded geometry metric spaces, and the morphisms to be coarse equivalences.  

Define a category $\mathcal{C}$ to be the quotient category of $\mathcal{C}_0$ under the equivalence relation on morphisms defined by closeness.
\end{corcat}

\begin{roecat}\label{roecat}
Let $X$ and $Y$ be spaces, and $U,V:l^2(X,\h)\to l^2(Y,\h)$ be unitary isomorphisms.  $U$ and $V$ are said to be \emph{close} if $U^*V$ is a finite propagation operator on $l^2(X,\h)$.

Define a category $\mathcal{R}_0$ by setting the objects to be uniformly discrete, bounded geometry metric spaces, and the morphisms to be unitary isomorphisms\footnote{Using lemma \ref{spatiallem}, it would be equivalent to define the morphisms to be $*$-algebra isomorphisms from $\C[X]$ to $\C[Y]$.} $U:l^2(X,\h)\to l^2(Y,\h)$ such that $T\mapsto UTU^*$ defines an isomorphism $\C[X]\to \C[Y]$.  

Define a category $\mathcal{R}$ to be the quotient category of $\mathcal{R}_0$ under the equivalence relation on morphisms defined by closeness. 
\end{roecat}

The following definition relates morphisms in $\mathcal{C}_0$ to morphisms in $\mathcal{R}_0$.

\begin{cover}\label{cover}
Let $f:X\to Y$ be a coarse equivalence, and $U:l^2(X,\h)\to l^2(Y,\h)$ a unitary isomorphism.  $U$ is said to \emph{cover} $f$ if there exists $C\geq 0$ such that for all $x\in X$ and $y\in Y$, $U_{yx}\neq 0$ implies that $d(f(x),y)\leq C$.
\end{cover}

\begin{covlem}\label{covlem}
\begin{enumerate}
\item For any morphism $U$ in $\mathcal{R}_0$, there exists a morphism $f$ in $\mathcal{C}_0$ covered by $U$. 
\item For any morphism $f$ in $\mathcal{C}_0$, there exists a morphism $U$ in $\mathcal{R}_0$ that covers $f$.  
\end{enumerate}
Moreover, any two morphisms in $\mathcal{R}_0$ covering the same morphism in $\mathcal{C}_0$ are close, and any two morphisms in $\mathcal{C}_0$ covered by a single morphism in $\mathcal{R}_0$ are close.
\end{covlem}

\begin{proof}
For the first part, let $f$ be built from $U$ as in the proof of Theorem \ref{athe}; clearly $U$ covers $f$.

The second part is folklore, but does not seem to appear in the literature: this is perhaps due to $K$-theoretic arguments requiring only covering \emph{isometries}, not covering \emph{unitaries}.  We sketch a proof.   Let $f:X\to Y$ be a morphism in $\mathcal{C}_0$.  Consider any partition $Y=\sqcup Y_n$ with the properties:
\begin{itemize}
\item there exists $C>0$ such that diam$(Y_n)\leq C$ for all $n$;
\item for all $n$, $Y_n\cap \text{Image}(f)\neq \emptyset$.
\end{itemize}
Using the fact that $f$ is a coarse equivalence, it is not difficult to see that such a partition exists.  Define a partition $X=\sqcup X_n$ by setting $X_n=f^{-1}(Y_n)$, and fix arbitrary unitary isomorphisms $U_n:l^2(X_n,\h)\to l^2(Y_n,\h)$ (such exist as the domain and codomain are separable infinite dimensional Hilbert spaces).  Define $U$ to be the unitary
$$
U:=\oplus U_n:l^2(X,\h)\to l^2(Y,\h);
$$
it is not hard to see that $U$ is a morphism in $\mathcal{R}_0$ that covers $f$.

The final comments follow from simple computations.
\end{proof}

Given the preceding lemma, the proof of the following theorem is a series of routine checks, and is thus omitted.

\begin{equivthe}\label{equivthe}
Provisionally define functors
$$
\mathcal{F}:\mathcal{R}\to\mathcal{C} ~~\text{ and }~~\mathcal{U}:\mathcal{C}\to\mathcal{R}
$$
as follows: each is the identity on objects; $\mathcal{F}[U]$ is the class in $\mathcal{C}$ of any coarse equivalence covered by $U$;  $\mathcal{U}[f]$ is the class in $\mathcal{R}$ of any unitary isomorphism covering $f$.

Then $\mathcal{F}$ and $\mathcal{U}$ are well-defined and mutually inverse isomorphisms of categories. \qed
\end{equivthe}

\subsection*{Acknowledgements}

This piece grew out of questions asked (independently) of us by our doctoral advisors: Guoliang Yu for \v{S}pakula and John Roe for Willett.  We are very grateful to Yu and Roe for suggesting this line of research, and for their ongoing encouragement and support.

We would also like to thank Jintao Deng for pointing out a gap in an earlier version of Lemma \ref{cloudbound1}; Romain Tessera for suggesting  the use of metric sparsification in Section \ref{fadsec}; and Christian Voigt for suggesting the use of Lemma \ref{spatiallem}.

The first author was supported by the \emph{Deutsche
Forschungsgemeinschaft} (SFB 878).  The second author was supported by the US NSF (DMS 1101174).


\begin{thebibliography}{10}

\bibitem{Arzhantseva:2008bv}
G.~Arzhantseva and T.~Delzant.
\newblock Examples of random groups.
\newblock Available on the authors' websites, 2008.

\bibitem{Brodzki:2011fk}
J.~Brodzki, G.~Niblo, J.~\v{S}pakula, R.~Willett, and N.~Wright.
\newblock Uniform local amenability.
\newblock {\em J. Noncommut. Geom.}, 7:583--603, 2013.

\bibitem{Brodzki:2007mi}
J.~Brodzki, G.~Niblo, and N.~Wright.
\newblock Property {A}, partial translation structures and uniform embeddings
  in groups.
\newblock {\em J. London Math. Soc.}, 76(2):479--497, 2007.

\bibitem{Lawrence-G.-Brown:1977co}
L.~G. Brown, P.~Green, and M.~A. Rieffel.
\newblock Stable isomorphism and strong morita equivalence of ${C}^*$-algebras.
\newblock {\em Pacific J. Math.}, 71(2):349--363, 1977.

\bibitem{Brown:2008qy}
N.~Brown and N.~Ozawa.
\newblock {\em ${C}^*$-Algebras and Finite-Dimensional Approximations},
  volume~88 of {\em Graduate Studies in Mathematics}.
\newblock American Mathematical Society, 2008.

\bibitem{Chen:2008so}
X.~Chen, R.~Tessera, X.~Wang, and G.~Yu.
\newblock Metric sparsification and operator norm localization.
\newblock {\em Adv. Math.}, 218(5):1496--1511, 2008.

\bibitem{Chen:2004bd}
X.~Chen and Q.~Wang.
\newblock Ideal structure of uniform {R}oe algebras of coarse spaces.
\newblock {\em J. Funct. Anal.}, 216(1):191--211, 2004.

\bibitem{Chen:2005cy}
X.~Chen and Q.~Wang.
\newblock Ghost ideals in uniform {R}oe algebras of coarse spaces.
\newblock {\em Arch. Math. (Basel)}, 84(6):519--526, 2005.

\bibitem{Connes:1994zh}
A.~Connes.
\newblock {\em Noncommutative Geometry}.
\newblock Academic Press, 1994.

\bibitem{Dixmier:1977vl}
J.~Dixmier.
\newblock {\em ${C^*}$-Algebras}.
\newblock North Holland Publishing Company, 1977.

\bibitem{Gromov:1993tr}
M.~Gromov.
\newblock Asymptotic invariants of infinite groups.
\newblock In G.~Niblo and M.~Roller, editors, {\em Geometric Group Theory},
  volume~2, 1993.

\bibitem{Gromov:2003gf}
M.~Gromov.
\newblock Random walks in random groups.
\newblock {\em Geom. Funct. Anal.}, 13(1):73--146, 2003.

\bibitem{Guentner:2005xr}
E.~Guentner, N.~Higson, and S.~Weinberger.
\newblock The {N}ovikov conjecture for linear groups.
\newblock {\em Publ. Math. Inst. Hautes \'{E}tudes Sci.}, 101:243--268, 2005.

\bibitem{Guentner:2022hc}
E.~Guentner and J.~Kaminker.
\newblock Exactness and the {N}ovikov conjecture.
\newblock {\em Topology}, 41(2):411--418, 2002.

\bibitem{Higson:1999km}
N.~Higson.
\newblock Counterexamples to the coarse {B}aum-{C}onnes conjecture.
\newblock Available on the author's website, 1999.

\bibitem{Higson:2002la}
N.~Higson, V.~Lafforgue, and G.~Skandalis.
\newblock Counterexamples to the {B}aum-{C}onnes conjecture.
\newblock {\em Geom. Funct. Anal.}, 12:330--354, 2002.

\bibitem{Higson:1995fv}
N.~Higson and J.~Roe.
\newblock On the coarse {B}aum-{C}onnes conjecture.
\newblock {\em London Mathematical Society Lecture Notes}, 227:227--254, 1995.

\bibitem{Higson:1993th}
N.~Higson, J.~Roe, and G.~Yu.
\newblock A coarse {M}ayer-{V}ietoris principle.
\newblock {\em Math. Proc. Cambridge Philos. Soc.}, 114:85--97, 1993.

\bibitem{Ioana:2010uq}
A.~Ioana.
\newblock {W}*-superrigidity for {B}ernoulli actions of property ({T}) groups.
\newblock Available on the author's website, 2010.

\bibitem{Ozawa:2000th}
N.~Ozawa.
\newblock Amenable actions and exactness for discrete groups.
\newblock {\em C. R. Acad. Sci. Paris S{\'e}r. I Math.}, 330:691--695, 2000.

\bibitem{Popa:2007fk}
S.~Popa.
\newblock Deformation and rigidity for group actions and von {N}eumann
  algebras.
\newblock In {\em Proceedings of the International Congress of Mathematicians},
  volume~I, pages 445--477. Eur. Math. Soc., 2007.

\bibitem{Roe:1988qy}
J.~Roe.
\newblock An index theorem on open mainfolds, {I}.
\newblock {\em J. Differential Geometry}, 27:87--113, 1988.

\bibitem{Roe:1993lq}
J.~Roe.
\newblock Coarse cohomology and index theory on complete {R}iemannian
  manifolds.
\newblock {\em Mem. Amer. Math. Soc.}, 104(497), 1993.

\bibitem{Roe:1996dn}
J.~Roe.
\newblock {\em Index Theory, Coarse Geometry and Topology of Manifolds},
  volume~90 of {\em CBMS Conference Proceedings}.
\newblock American Mathematical Society, 1996.

\bibitem{Roe:2003rw}
J.~Roe.
\newblock {\em Lectures on Coarse Geometry}, volume~31 of {\em University
  Lecture Series}.
\newblock American Mathematical Society, 2003.

\bibitem{Roe:2005rt}
J.~Roe.
\newblock Hyperbolic groups have finite asymptotic dimension.
\newblock {\em Proc. Amer. Math. Soc.}, 133(9):2489--2490, 2005.

\bibitem{Sako:2012fk}
H.~Sako.
\newblock Property {A} and the operator norm localization property for discrete
  metric spaces.
\newblock arXiv:1203.5496, 2012.

\bibitem{Skandalis:2002ng}
G.~Skandalis, J.-L. Tu, and G.~Yu.
\newblock The coarse {B}aum-{C}onnes conjecture and groupoids.
\newblock {\em Topology}, 41:807--834, 2002.

\bibitem{Spakula:2009rr}
J.~\v{S}pakula.
\newblock Non-${K}$-exact uniform {R}oe ${C^*}$-algebras.
\newblock To appear in Journal of ${K}$-theory; available on the author's
  website, 2009.

\bibitem{Spakula:2009tg}
J.~\v{S}pakula.
\newblock Uniform ${K}$-homology theory.
\newblock {\em J. Funct. Anal.}, 257(1):88--121, 2009.

\bibitem{Wang:2007uf}
Q.~Wang.
\newblock Remarks on ghost projections and ideals in the {R}oe algebras of
  expander sequences.
\newblock {\em Arch. Math. (Basel)}, 89(5):459--465, 2007.

\bibitem{Whyte:1999uq}
K.~Whyte.
\newblock Amenability, bi-{L}ipschitz equivalence and the von {N}eumann
  conjecture.
\newblock {\em Duke Math. J.}, 99(1):93--112, 1999.

\bibitem{Willett:2009rt}
R.~Willett.
\newblock Some notes on property {A}.
\newblock In G.~Arzhantseva and A.~Valette, editors, {\em Limits of Graphs in
  Group Theory and Computer Science}, pages 191--281. EPFL press, 2009.

\bibitem{Winter:2010eb}
W.~Winter and J.~Zacharias.
\newblock The nuclear dimension of ${C}^*$-algebras.
\newblock {\em Adv. Math.}, 224(2):461--498, 2010.

\bibitem{Yu:1995bv}
G.~Yu.
\newblock Coarse {B}aum-{C}onnes conjecture.
\newblock {\em ${K}$-theory}, 9:199--221, 1995.

\bibitem{Yu:1998wj}
G.~Yu.
\newblock The {N}ovikov conjecture for groups with finite asymptotic dimension.
\newblock {\em Ann. of Math.}, 147(2):325--355, 1998.

\bibitem{Yu:200ve}
G.~Yu.
\newblock The coarse {B}aum-{C}onnes conjecture for spaces which admit a
  uniform embedding into {H}ilbert space.
\newblock {\em Invent. Math.}, 139(1):201--240, 2000.

\end{thebibliography}

\end{document}